\documentclass[11pt,leqno]{amsart}
\usepackage{amsthm,amsfonts,amssymb,amsmath,oldgerm}
\usepackage{graphicx}
\usepackage{makecell}
\numberwithin{equation}{section}
\usepackage{fullpage}

\usepackage{hyperref}
\allowdisplaybreaks[2]

\renewcommand\d{\partial}

\renewcommand\b{\beta}

\def\eps{\varepsilon }

\renewcommand\d{\partial}

\renewcommand\b{\beta}

\newcommand\R{\mathbb R}
\newcommand\C{\mathbb C}

\def\eps{\varepsilon}

\newcommand\br{\begin{remark}}
\newcommand\er{\end{remark}}
\newcommand\bp{\begin{pmatrix}}
\newcommand\ep{\end{pmatrix}}
\newcommand{\be}{\begin{equation}}
\newcommand{\ee}{\end{equation}}
\newcommand\ba{\begin{equation}\begin{aligned}}
\newcommand\ea{\end{aligned}\end{equation}}

\newcommand{\bap}{\begin{app}}
\newcommand{\eap}{\end{app}}
\newcommand{\begs}{\begin{exams}}
\newcommand{\eegs}{\end{exams}}
\newcommand{\beg}{\begin{example}}
\newcommand{\eeg}{\end{example}}
\newcommand{\bpr}{\begin{proposition}}
\newcommand{\epr}{\end{proposition}}
\newcommand{\bt}{\begin{theorem}}
\newcommand{\et}{\end{theorem}}
\newcommand{\bc}{\begin{corollary}}
\newcommand{\ec}{\end{corollary}}
\newcommand{\bl}{\begin{lemma}}
\newcommand{\el}{\end{lemma}}
\newcommand{\bd}{\begin{definition}}
\newcommand{\ed}{\end{definition}}
\newcommand{\brs}{\begin{remarks}}
\newcommand{\ers}{\end{remarks}}

\newcommand{\const}{\text{\rm constant}}
\newcommand{\Id}{{\rm Id }}

\DeclareMathOperator{\sgn}{sgn}
\newtheorem{theorem}{Theorem}[section]
\newtheorem{proposition}[theorem]{Proposition}
\newtheorem{corollary}[theorem]{Corollary}
\newtheorem{lemma}[theorem]{Lemma}

\theoremstyle{remark}
\newtheorem{remark}[theorem]{Remark}
\theoremstyle{definition}
\newtheorem{definition}[theorem]{Definition}

\newtheorem{example}[theorem]{Example}

\newcommand{\beq}{\begin{equation}}
\newcommand{\eeq}{\end{equation}}

\title{Spectral stability of hydraulic shock profiles}
\author{Alim Sukhtayev}
\address{Miami University, Oxford, OH 45056}
\email{sukhtaa@miamioh.edu}
\thanks{}

\author{Zhao Yang}
\address{Indiana University, Bloomington, IN 47405}
\email{yangzha@indiana.edu}
\thanks{Research of Z.Y. was supported by the College of Arts and Sciences Dissertation Year Fellowship}
\author{Kevin Zumbrun}
\address{Indiana University, Bloomington, IN 47405}
\email{kzumbrun@indiana.edu}
\thanks{Research of K.Z. was partially supported
under NSF grant no. DMS-1400555}
\begin{document}

\begin{abstract}
By reduction to a generalized Sturm Liouville problem, we establish spectral stability of
hydraulic shock profiles of the Saint-Venant equations for inclined shallow-water flow,
over the full parameter range of their existence, for both smooth-type profiles
and discontinuous-type profiles containing subshocks.  Together with work of Mascia-Zumbrun and
Yang-Zumbrun, this
yields linear and nonlinear $H^2\cap L^1 \to H^2$ stability with sharp rates of decay in $L^p$, $p\geq 2$,
the first complete
stability results for large-amplitude shock profiles 
of a hyperbolic relaxation system. 
\end{abstract}

\date{\today}
\maketitle



\section{Introduction}\label{s:intro}
In this paper, building on work of \cite{YZ,JNRYZ}, we study spectral stability of hydraulic shock profiles 
of the (inviscid) Saint-Venant equations for inclined shallow-water flow:
\ba \label{sv}
\d_th+\d_xq&=0,\\
\d_tq+\d_{x}\left(\frac{q^2}{h}+\frac{h^2}{2F^2}\right)&=h-\frac{|q|q}{h^2},
\ea
where $h$ denotes fluid height; $q=hu$ total flow, with $u$ fluid velocity; and
$F>0$ the {\it Froude number}, a nondimensional parameter depending on reference height/velocity and inclination.

Equations \eqref{sv} are the standard ones used in the hydraulic engineering literature to describe flow in a dam spillway 
or other inclined channel; see \cite{BM,Je,Br1,Br2,Dr,JNRYZ,YZ} and references therein.
They have the form of a $2\times 2$ {\it hyperbolic system of balance laws} \cite{L,W,Bre,Da},
with relaxation terms $h-\frac{|q|q}{h^2}$ on the righthand side of \eqref{sv}(ii) representing the balance between
gravitational force and turbulent bottom friction 
(modeled following Chezy's formula as proportional to velocity squared \cite{Dr,BM}).
The associated {\it equlibrium} (or ``relaxed'') {\it model}, obtained by setting $q=q_*(h):= h^{3/2}$ so that gravity and
friction exactly cancel, is the {\it scalar conservation law}
\be\label{CE}
\d_t h + \d_x  q_*(h)=0,
\ee
a generalized Burgers equation.

As noted by Jeffreys \cite{Je}, there is an important distinction between the {\it hydrodynamically stable} case $0<F<2$
and the {\it hydrodynamically unstable} case $F>2$.
In the former case, the subcharacteristic condition of Whitham is satisfied \cite{W,L}, and constant, equilibirum solutions
$(h,q)\equiv (h_0, q_*(h_0)$ of \eqref{sv} are stable under perturbation (the definition of hydrodynamic
stability); moreover, the behavior under nonlinear perturbation is approximately governed by \eqref{CE}.
For $F>2$, constant solutions are always unstable and behavior is quite different, featuring pattern formation and
onset of complex dynamics \cite{Dr}.
Indeed, this dichotomy between hydrodynamically stable and unstable regimes
is typical of general relaxation systems \cite{W,L,JK}.

Following \cite{YZ}, we here focus on the hydrodynamically stable case $0<F<2$, and associated {\it hydraulic shock profile} solutions 
\be\label{prof}
(h,q)(x,t)= (H,Q)(x-ct), \quad \lim_{z\to - \infty}(H,Q)(z)= (H_L,Q_L), \; \lim_{z\to - \infty}(H,Q)(z)= (H_R,Q_R), 
\ee
analogous to shock wave solutions of the equlibrium system \eqref{CE}.
These are piecwise smooth traveling-wave solutions satisfying the Rankine-Hugoniot jump and Lax entropy conditions 
\cite{Sm,Da,Bre,La} at any discontinuities. Their existence theory reduces to the study of an explicitly solvable 
scalar ODE with polynomial coefficients \cite{YZ}; it is described completely as follows.

\bpr[\cite{YZ}]\label{existprop}
Let $(H_L,H_R, c)$ be a triple for which there exists an entropy-admissible shock solution in the sense of Lax \cite{La}
with speed $c$ of \eqref{CE} connecting left state $H_L$ to right state $H_R$, i.e., $H_L>H_R>0$ and
$c[H]=[q_*(H)]$.
Then, there exists a corresponding hydraulic shock profile \eqref{prof} with $Q_L=q_*(H_L)$ and $Q_R=q_*(H_R)$
precisely if $0<F<2$.
The profile is smooth for $H_L> H_R> H_L \frac{2F^2}{1+2F+\sqrt{1+4F}}$, and nondegenerate in the sense that $c$ is not a
characteristic speed of \eqref{sv} at any point along the profile.
For $0<H_R< H_L \frac{2F^2}{1+2F+\sqrt{1+4F}}$, the profile is nondegenerate and piecewise smooth, with a single
discontinuity consisting of an entropy-admissible shock of \eqref{sv}.
At the critical value $H_R= H_L \frac{2F^2}{1+2F+\sqrt{1+4F}}$, $H_R$ is characteristic,
and there exists a degenerate profile that is continuous but not smooth, with discontinuous derivative at $H_R$.
\epr

Typical profiles of each type (smooth, degenerate, piecewise smooth) are displayed in Figure \ref{profile}.

\begin{figure}
\begin{center}
\includegraphics[scale=0.32]{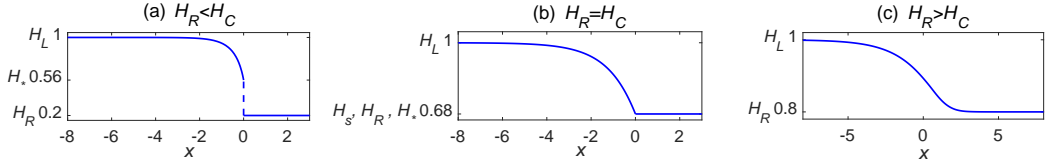}
\end{center}
\caption{Hydraulic shock profiles with $F=1.5$, $H_L=1$ and (a) $H_R=0.2$; (b) $H_R=\frac{9}{8+2\sqrt{7}}$; (c) $H_R=0.8$,
reproduced from \cite{YZ} with permission of the authors.}
\label{profile}
\end{figure}

\subsection{Main results}\label{s:results} We now turn to the discussion of stability, and our main results.
Linearizing \eqref{sv} about a smooth profile $(H,Q)$ following \cite{MZ}, we obtain eigenvalue equations
\be 
\label{syseval}
Av'=(E-\lambda \Id-A_x)v,
\ee
where
\ba
A&=\left[\begin{array}{cc} -c & 1\\ \frac{H}{F^2}-\frac{Q^2}{H^2} & \frac{2Q}{H}-c \end{array}\right],\quad E=\left[\begin{array}{cc} 0 & 0\\ \frac{2Q^2}{H^3}+1 & -\frac{2Q}{H^2} \end{array}\right].\\
\ea
It is shown in \cite{YZ} that essential spectrum of $\mathcal{L}:= -A\partial_x- \partial_xA +E$ is confined to
$\{\lambda:  \Re \lambda <0\}\cup \{0\}$, with an embedded eigenvalue at $\lambda=0$.
Moreover, it is shown that the embedded eigenvalue at $\lambda=0$ is of multiplicity one in a 
generalized sense defined in terms of an associated Evans function defined as in \cite{AGJ,GZ,MZ}.
It follows by the general theory of \cite{MZ2} relating generalized, or Evans-type, spectral stability to linearized and nonlinear stability, that smooth hydraulic shock profiles are nonlinearly orbitally stable so long as they are
{\it weakly spectrally stable} in the sense that there exist no decaying solutions of \eqref{syseval}
on $\{\lambda:  \Re \lambda \geq 0\}\setminus \{0\}$.

The discontinuous case is more complicated, involving a free boundary with transmission/evolution conditions given by the Rankine-Hugoniot jump conditions.
However, following the approach of Erpenbeck-Majda \cite{Er1,Er2,Ma} for the study of such problems in the context
of shocks and detonations, one may deduce a generalized eigenproblem consisting of the same ODE \eqref{syseval},
but posed on the negative half-line $x\in (-\infty,0)$ with boundary condition 
\be \label{bc}
[\lambda\overline{W}-R(\overline{W})]_\perp \cdot A(0^-)v(0^-)=0,
\ee
where $\overline{W}:=(H,Q)^T$ and $[h]:= h(0^+)-h(0^-)$ denotes jump in $h$ across $x=0$; see \cite{YZ} for further details.
Similarly as in the smooth case, it is shown in \cite{YZ} that
essential spectrum of $\mathcal{L}$ with boundary condition \eqref{bc}
is confined to $\{\lambda:  \Re \lambda <0\}\cup \{0\}$, with an embedded eigenvalue at $\lambda=0$, of multiplicity
one in a generalized sense defined by an associated Evans-Lopatinsky function.
It follows by the general theory of \cite{YZ} that discontinuous hydraulic shock profiles are 
nonlinearly orbitally stable so long as they are weakly spectrally stable in the sense that there exist 
no decaying solutions of \eqref{syseval}-\eqref{bc} on $\{\lambda:  \Re \lambda \geq 0\}\setminus \{0\}$.

In summary, by the analytical results of \cite{MZ2,YZ}, the question of nonlinear stability of hydraulic shock profiles 
has been reduced in both smooth and discontinuous case 
to determination of weak spectral stability, or nonexistence of eigenvalues $\lambda\neq 0$
with $\Re \lambda \geq 0$ of eigenvalue problem \eqref{syseval} on the whole- or half-line, respectively.
The weak spectral stability condition was verified numerically in \cite{YZ} for both smooth and piecewise smooth profiles
by extensive Evans/Evans-Lopatinsky function computations across the entire parameter range of existence,
indicating linearized and nonlinear stability.  However, the computation was done with ordinary machine
rather than interval arithmetic, and this conclusion though decisive falls short of rigorous proof.

In the present work, we establish the following theorem verifying analytically the conclusions obtained numerically
in \cite{YZ}, from which nonlinear stability then follows by the results of \cite{MZ2,YZ}.

\bt \label{main}
Nondegenerate hydraulic shock profiles of the Saint-Venant equations \eqref{sv} are weakly spectrally stable,
across the entire range of existence described in Proposition \ref{existprop}.
\et

\begin{proof}
This follows by Corollaries \ref{smoothstab} and \ref{discontstab} below.
\end{proof}

\bc[\cite{MZ2,YZ}]\label{maincor}
Nondegenerate hydraulic shock profiles of \eqref{sv} are linearly and nonlinearly orbitally stable.
Specifically, let $\overline{W}=(H,Q)$ be a hydraulic shock profile \eqref{prof},
and $v_0$ be an initial perturbation supported away from any discontinuity of $\overline{W}$ and
of norm $\eps$ sufficiently small in $H^{s}\cap L^1$, $s\geq 2$. 
Then, for initial data $\tilde W_0:=\overline{W}_0+v_0$, there exists a global solution $\tilde W$ of \eqref{sv}
and a phase shift $\eta$, satisfying for $2\leq p\leq \infty$: 
\ba\label{mainests}
\|\tilde W(\cdot,t)-\overline{W}(\cdot-ct +\eta(t))\|_{H^s}&\leq C\eps (1+t)^{-1/4},\\
\|\tilde W(\cdot,t)-\overline{W}(\cdot-ct +\eta(t))\|_{L^p}&\leq C\eps (1+t)^{-(1/2)(1-1/p)}, \\
|\dot \eta(t)|&\leq C\eps (1+t)^{-(1/2)}.
\ea
\ec

Theorem \ref{main} and Corollary \ref{maincor} together represent the first complete analytical stability
result for large-amplitude shock profiles of a quasilinear relaxation system\footnote{
Profiles of the semilinear Jin-Xin relaxation model \cite{JX} are stable for arbitrary
amplitude, by $L^1$-contraction/comparison \cite{MZ}.}
and the first for discontinuous shock profiles of a relaxation system of any kind.

\subsection{Discussion and open problems}\label{s:discussion}
A general approach to stability of traveling waves in systems of conservation and balance laws
is the ``divide and conquer'' algorithm described in, e.g., \cite{Z1,Z2,Z3}, wherein ``Lyapunov-type'' theorems relating
spectral to linearized and nonlinear stability are established in a very general setting, then spectral stability is
verified in a problem-specific way, whether by numerics, asymptotics, or special structure of the equations.
Useful topological criteria involving various ``stability indices'' are often explicitly computable
as necessary conditions, leading to analytical {\it instability} results for large-amplitude waves of quite general systems \cite{GZ,Z1,Z2}.
By comparison, complete {\it global stability} results as in, e.g., \cite{CGS,JX,Z4,MW,HLZ,LW}, are quite rare, 
exploiting special nonlinear structure of the system under consideration.

The special structure exploited here is that the eigenvalue system \eqref{syseval} may be reduced to a scalar second-order system of generalized
Sturm-Liouville type.
Specifically, following the general approach described in Section \ref{s:red}, the eigenvalue system \eqref{syseval} originating from any $2\times 2$
relaxation system may converted to a scalar second-order equation
\be\label{geneval}
Lw= (\lambda \alpha(x)w + \lambda^2 \beta(x)) w, \quad w\in \C,
\ee
where $L$ is a (fixed) second-order scalar operator with real-valued coefficients and $\alpha$ and $\beta$ are real.
For the specific case treated here, we find that, by a further Liouville transformation, we may arrange that {\it $Lw=w'' +q(x)w$ is self-adjoint and $\alpha$ and $\beta$ are strictly positive:}
the generalized Sturm-Liouville structure to which we refer above.
In the half-line case, there is in addition a $\lambda$-dependent Robin-type boundary condition
\be\label{robin}
w'(0)=(c + \phi(\lambda))w(0), \quad \phi(0)=0,
\ee
for which we find $\Im \phi(\lambda),  \Re \phi(\lambda)\leq 0$ for $\Re \lambda \geq 0$.
From this structure, together with monotonicity of the underlying traveling wave, 
we are able to deduce stability by 
a combination of standard Sturm-Liouville principles and ``by-hand'' computation.

This argument, while decisively answering the question of stability of hydraulic shocks, at the same time suggests a number
of other interesting questions.
For example, given the complexity of the formulae involved, to arrive at the end of computations 
to the above-described special structure appears little short of miraculous.
Is this a lucky accident?  Or is it somehow forced by the properties of the wave?
More generally, given a generalized eigenvalue problem \eqref{geneval} for which all eigenvalues are stable, is there 
some choice of coordinate system in which the resulting $\alpha$ and $\beta$ are strictly positive?
And, still more generally, what are the minimum structural requirements under which one can recover a full
or partial suite of standard Sturm-Liouville results?

Finally, we pose the question, open so far as we know, whether shock profiles of general 
$2\times 2$ relaxation systems of the type considered in \cite{L} are always stable, or whether one can
find examples of spectrally unstable smooth or discontinuous profiles for amplitudes sufficiently large.


\section{Profiles and reduction to second order scalar ODE}

\subsection{Profiles} Following \cite[\S 2]{YZ}, we find, substituting the ansatz \eqref{prof} into \eqref{sv}
and using the first (conservative) equation to eliminate $Q$, that profiles satisfy on smooth regimes the
first-order scalar traveling-wave ODE
\be
\label{profileODE}
H'=\frac{F^2 \left(H - H_L\right) \left(H - H_R\right) \left(H-H_3\right)}{(H-H_s)(H^2+HH_s+H_s^2)}
\ee
where
\be
\label{H3Hs}
H_3:=\frac{\nu^2}{\nu^2+2\nu+1}H_R,\quad H_s:=\left(\frac{F\nu^2}{\nu+1}\right)^{\frac{2}{3}}H_R,
\quad \nu:=\sqrt{\frac{H_L}{H_R}}>1,
\ee
with $Q$ determined (from the first equation) by
\be\label{Q}
Q-c H\equiv \const =:-q_0.
\ee
From $\nu>1$, we have $H_3<H_R<H_L$.
When $F<2$, we have also $H_s<H_L$. When also $H_s<H_R$, there exists a smooth profile connecting equilibria
$H_L$ and $H_R$; when $H_R<H_s<H_L$, there exists a fifth point
\be
\label{Hstar}
H_*:=\frac{-\nu-1+\sqrt{8F^2\nu^4+\nu^2+2\nu+1}}{2\left(\nu+1\right)}H_R
\ee
lying between $H_s$ and $H_L$, such that there is an entropy-admissible Lax shock of \eqref{sv} from $H_*$
to $H_R$, hence a discontinuous hydraulic shock profile connecting equilibrium $H_L$ smoothly to the nonequilibrium point
$H_*$, and then by a Lax shock from $H_*$ to equilibrium $H_R$.

In both cases, the rational-coefficient ODE \eqref{profileODE} may be solved explicitly for $H$ as a function of $x$.
However, by monotonicity of $H$ on smooth parts of the profile (as holds for any scalar ODE),
we may equally well change coordinates and take $H$ as independent variable in place of $x$, as in 
\cite{JNRYZ,YZ}. Thus, we do not need anywhere the precise form of $(H,Q)(x)$ in our analysis here.

By the rescaling introduced in \cite{YZ} [Obeservation 2.4], we can fix $H_L=1$ and $0<H_R<1$. From now on, we substitute $H_L=1$ and assume $0<H_R<1$ in our analysis.


\subsection{Reduction to second-order scalar form}\label{s:red}
By performing a change of unknowns, we may rewrite the system \eqref{syseval} as a second order scalar ODE. Consider first a general $2 \times  2$ system of ODE
\be
\underbrace{
\left[\begin{array}{rr}
p_{11}(x)&p_{12}(x)\\
p_{21}(x)&p_{22}(x)
\end{array}
\right]}_{P}
\left[\begin{array}{rr}v_1(x)\\
v_2(x)
\end{array}\right]'=
\underbrace{\left[\begin{array}{rr}
q_{11}(x)&q_{12}(x)\\
q_{21}(x)&q_{22}(x)
\end{array}
\right]}_{Q}
\underbrace{\left[\begin{array}{rr}v_1(x)\\
v_2(x)
\end{array}\right]
}_{v}
\ee
with $p_{12}(0)\neq 0$. Let $T_1(x)=\left[\begin{array}{rr}1&0\\-\frac{p_{22}(x)}{p_{12}(x)} &1\end{array}\right]$, $T_2(x)=\left[\begin{array}{rr}1&0\\-\frac{p_{11}(x)}{p_{12}(x)} &1\end{array}\right]$ and note that $T_1PT_2=\left[\begin{array}{rr}0&p_{12}\\-\frac{\det P}{p_{12}} &0\end{array}\right]$. Defining the change of unknowns $v=T_2u$, the above becomes
$P(T_2u'+T_2'u)=QT_2u$.  Left multiplying $T_1$ on both hand sides, we have $T_1PT_2u'=(T_1QT_2-T_1PT_2')u=:Mu$, that is,
\be\label{veq}
\left[\begin{array}{rr}0&p_{12}\\-\frac{\det P}{p_{12}} &0\end{array}\right]\left[\begin{array}{rr}u_1\\
u_2
\end{array}\right]^\prime=\left[\begin{array}{rr}m_{11}& m_{12}\\m_{21}&m_{22}\end{array}\right]\left[\begin{array}{rr}u_1\\
u_2
\end{array}\right]
\ee
where here $M=[m_{i,j}]_{i,j=1}^2$.
Assuming $m_{11}(0)\neq 0$, the first equation $p_{12}u_2'=m_{11}u_1+m_{12}u_2$
yields $u_1=\frac{p_{12}u_2'-m_{12}u_2}{m_{11}}$.
Substituting in the second equation, we obtain the second-order scalar ODE
\be
-\frac{\det P}{p_{12}}\Big(\frac{p_{12}u_2'-m_{12}u_2}{m_{11}}\Big)^\prime=m_{21}\frac{p_{12}u_2'-m_{12}u_2}{m_{11}}+m_{22}u_2.
\ee 
Specialized to system \eqref{syseval}, by setting 
\be\label{T2}
v=T_2u, \qquad T_2=\left[\begin{array}{cc} 1 & 0\\ c & 1\end{array}\right],\\
\ee
and following the reduction procedures, the eigenvalue system \eqref{syseval} reduces for $\lambda \neq 0$ to 
\be\label{u2eq} 
u_2''+(f_1\lambda+f_2)u_2'+(f_3\lambda^2+f_4\lambda)u_2=0,
\qquad
u_1= - \frac{u_2'}{\lambda},
\ee
where $f_i$, $i=1,\dots,4$ are explicitly computable functions. (We display various combinations below where they are helpful,
but in general these are lengthy and we do not give them here.)
In terms of the original coordinates,
\be\label{origui}
u_1=h, \qquad u_2= q-ch.
\ee

For later use, we note that, dividing by $(f_3\lambda^2+f_4\lambda)$, differentiating, substituting
$-\lambda u_1$ everywere for $u_2'$, and rearranging, we may write \eqref{u2eq} alternatively as 
\be\label{u1eq}
\Big( \frac{u_1'+(f_1\lambda+f_2)u_1}{ f_3\lambda+f_4}\Big)'= -\lambda u_1,
\qquad
u_2(x)= - \lambda \int_{-\infty}^x u_1(y)dy,
\ee
to obtain a formulation for which all eigenvalues agree with those of \eqref{syseval}, including the translational
eigenvalue at $\lambda=0$, corresponding to $(h,q)=(H',Q')$, or $(u_1, u_2)=(H', 0)$. 
The formulations \eqref{u1eq} and \eqref{u2eq} may be recognized respectively as analogous to the ``flux'' and 
``balanced flux'' formulations of \cite{PZ}, the latter of which has the advantage of removing the translational
eigenvalue at $\lambda=0$.
For all other eigenvalues on $\Re \lambda \geq 0$, the spectra of \eqref{syseval}, \eqref{u2eq}, and
\eqref{u1eq} agree.
In particular, setting $\lambda=0$ and recalling \eqref{origui}, we record that $\bar h:= H'$ satisfies
\be\label{hident}
\bar h'+ f_2 \bar h=0.
\ee

Introducing now the Liouville-type transformation 
$$
w(\lambda,x)=e^{\frac{1}{2}\int_0^{x}(f_1(y)\lambda+f_2(y))dy}u_2(\lambda,x),
$$
we find that $w$ satisfies
\be 
\label{weq}
w''+\left(\Big(f_3-\frac{1}{4}f_1^2\Big)\lambda^2+\Big(f_4-\frac{1}{2}f_1f_2-\frac{1}{2}f_1'\Big)\lambda-\frac{1}{4}f_2^2-\frac{1}{2}f_2'\right)w=0.
\ee 

{\bf Summary:} The eigenvalues of \eqref{syseval} and the generalized eigenvalue equation \eqref{weq}
agree for $\Re\lambda \geq 0$ and $\lambda\neq 0$, hence to establish weak spectral stability of 
hydraulic shock profiles, it is sufficient to show that {\it \eqref{weq} admits no eigenvalues on 
$\Re \lambda\geq 0$ other than $\lambda=0$.}
In fact we shall show that \eqref{weq} admits no eigenvalues on $\Re\lambda \geq 0$, that is, the translational
zero eigenvalue of the original problem has been removed by the coordinate transformation to variable $w$.

\section{Spectral stability of smooth hydraulic shock profiles}

In order for $w$ to decay exponentially at $\pm\infty$, it is required that 
\ba 
\label{rate}
\Re\gamma_{1,-}(\lambda)+\lim_{y\rightarrow -\infty}\frac{1}{2}\left(f_1(y)\Re \lambda+f_2(y)\right)>0,\quad \Re\gamma_{2,+}(\lambda)+\lim_{y\rightarrow +\infty}\frac{1}{2}\left(f_1(y)\Re\lambda+f_2(y)\right)<0.
\ea 
where $\gamma_{2,+}$, $\gamma_{1,-}$ defined in \cite{YZ}[(4.8) (4.9)] are the expected decaying rate of eigenmode $v(\lambda,x)$ as $x\rightarrow \pm\infty$.
Calculation shows 
\ba
\label{ratem}
&\gamma_{1,-}(\lambda)+\lim_{y\rightarrow -\infty}\frac{1}{2}\left(f_1(y)\lambda+f_2(y)\right)\\
=&\frac{Fv\left(v+1\right)\sqrt{4\lambda ^2v^2{\left(v+1\right)}^2+4\lambda v\left(v+1\right)\left(-F^2+2v^2+2v\right)+F^2{\left(v^2+v-2\right)}^2}}{2\left(-F^2+v^4+2v^3+v^2\right)}
\ea
\ba
\label{ratep}
&\gamma_{2,+}(\lambda)+\lim_{y\rightarrow +\infty}\frac{1}{2}\left(f_1(y)\lambda+f_2(y)\right)\\
=&-\frac{Fv\left(v+1\right)\sqrt{4\lambda ^2{\left(v+1\right)}^2+4\lambda v\left(v+1\right)\left(-F^2v^2+2v+2\right)+F^2v^2{\left(-2v^2+v+1\right)}^2}}{2\left(-F^2v^4+v^2+2v+1\right)}
\ea
\eqref{rate} then holds.

\begin{lemma}\label{noimaginaryeigen}
The system \eqref{syseval} has no nonzero pure imaginary eigenvalue.
\end{lemma}
\begin{proof}
By coordinate change $v\leftrightarrow u\leftrightarrow w$ and exponential decay of $w(x)$ as $x\rightarrow \pm\infty$ ensured by \eqref{rate}, existence of eigenmodes of \eqref{syseval} is equivalent to existence of exponential decaying (as $x\rightarrow\pm\infty$) solutions $w$ to \eqref{weq}. Let now $\lambda= ia, a\neq 0$ be an eigenvalue and $w$ a
corresponding decaying solution.
Substituting $\lambda=ia$, $w$ in \eqref{weq} implies
\be
\label{iaw}
w''+\Big(-a^2f_3+\frac{1}{4}a^2f_1^2-\frac{1}{4}f_2^2-\frac{1}{2}f_2'\Big)w=ia\Big(-f_4+\frac{1}{2}f_1f_2+\frac{1}{2}f_1'\Big) w.
\ee 
Taking the $L^2$ inner product of $w$ with \eqref{iaw} on the whole line yields
\be
\label{inner}
-\langle w',w'\rangle+\Big\langle w,\Big(-a^2f_3+\frac{1}{4}a^2f_1^2-\frac{1}{4}f_2^2-\frac{1}{2}f_2'\Big)w\Big\rangle=ia\Big\langle w,\Big(-f_4+\frac{1}{2}f_1f_2+\frac{1}{2}f_1'\Big) w\Big\rangle.
\ee 
Taking the imaginary part of \eqref{inner} then gives 
\be 
\Big\langle w,\Big(-f_4+\frac{1}{2}f_1f_2+\frac{1}{2}f_1'\Big) w\Big\rangle=0.
\ee
We will reach a contradiction provided that $\left(-f_4+\frac{1}{2}f_1f_2+\frac{1}{2}f_1'\right)$ has definite sign. But
\be
\label{defico}
-f_4+\frac{1}{2}f_1f_2+\frac{1}{2}f_1'=\frac{F^2\left(H-H_R+H(\sqrt{H_R}+H_R)\right)f_{F,H_R}(H)}{\left(\sqrt{H_R}+1\right)^3{\left(H^3-H_s^3\right)}^2},
\ee
\be 
\label{fFHR}
f_{F,H_R}(H):=2{\left(\sqrt{H_R}+1\right)}^2H^3-F^2H_R\left(H_R+\sqrt{H_R}+1\right)H+F^2{H_R}^2.
\ee 
It then suffices to show $f_{F,H_R}(H)$ has definite sign.
The positive critical point of $f_{F,H_R}(\cdot)$ is 
\be 
\label{Hcdef}
H_c(F,H_R)=\frac{F\sqrt{H_R(H_R+\sqrt{H_R}+1)}}{\sqrt{6}(\sqrt{H_R}+1)}.
\ee
Further, we have
\be 
H_c(F,H_R)<F\frac{\sqrt{H_R}}{\sqrt{H_R}+1}<H_R,
\ee 
in which the last inequality holds because the domain of existence of smooth hydraulic shocks is
\be\label{domain}
H_R+\sqrt{H_R}>F.
\ee
By monotonicity of $f_{F,H_R}(\cdot)$ \eqref{fFHR} on $[H_c,\infty)$, we thus have 
\ba
f_{F,H_R}(H)&>f_{F,H_R}(H_R)\\
&={H_R}^{5/2}\left(\sqrt{H_R}+1\right)\left(2H_R-F^2+2\sqrt{H_R}\right)\\
&>{H_R}^{5/2}\left(\sqrt{H_R}+1\right)\left(FH_R-F^2+F\sqrt{H_R}\right)\\
&={H_R}^{5/2}\left(\sqrt{H_R}+1\right)F\left(H_R-F+\sqrt{H_R}\right)> 0,
\ea
in which the last inequality holds, again, by \eqref{domain}.
\end{proof}

\begin{corollary}\label{smoothstab}

All smooth hydraulic shock profiles are weakly spectrally stable in the sense that system \eqref{syseval} has no eigenvalue $\lambda$ with $\Re \lambda \geq 0$ and $\lambda\neq0$.
\end{corollary}

\begin{proof}
By their characterization as roots of the Evans function, which is analytic on $\Re\lambda \geq -\eta$ for some
$\eta>0$, and real analytic in parameters $F$, $H_L$, $H_R$ \cite{MZ,MZ2,YZ}, we
see readily that eigenvalues associated with \eqref{syseval} perturb continuously as parameters are varied, in both
location and multiplicity. In particular, the fact shown in Lemma \ref{noimaginaryeigen} that there are no 
nonzero imaginary eigenvalues together with the fact shown in \cite{YZ} that there is an eigenvalue of 
fixed multiplicity one at $\lambda=0$, 
implies that no eigenvalues can cross from  $\Re \lambda<0$ to $\Re \lambda \geq 0$ as parameters are varied.
By connectedness of the parameter range on which hydraulic shock profiles exist, therefore, we find by a homotopy
argument that the number of nonstable eigenvalues, $\Re \lambda \geq 0$ is constant across the entire domain of
existence.
But, by \cite{MZ3}, small-amplitude hydraulic shock profiles are spectrally stable, hence have precisely one 
nonstable eigenvalue consisting of a simple root of the Evans function at $\lambda=0$.
Thus, the number of nonstable roots for all hydraulic shock profiles must be $1$, and this is accounted for by
the multiplicity one root at the origin corresponding to translational invariance of the underlying equations
\cite{MZ,MZ2}.
It follows that there are no nonstable eigenvalues other than $\lambda=0$, and all profiles are weakly spectrally stable
as claimed.
\end{proof}

\subsection{Alternate proof}\label{s:alternate}
We give also an alternate, direct proof of stability, both for its own interest and as practice for nonsmooth case.
Denote by
\be\label{L}
Lw:= w''+\Big(-\frac{1}{4}f_2^2-\frac{1}{2}f_2'\Big)w
\ee
the self-adjoint operator given by the $\lambda=0$ part of the lefthand side of \eqref{weq},
and denote by
\be\label{B}
\mathcal{B}(\tilde{w},w):=-\langle \tilde{w}',w'\rangle - \Big\langle \tilde{w}, \Big(\frac{1}{4}f_2^2+\frac{1}{2}f_2'\Big)w\Big\rangle
\ee
the bilinear form induced on $\tilde{w},\;w\in H^1(\R)$ by $\mathcal{B}(\tilde{w},w):=\langle \tilde{w}, Lw\rangle$.

\bl\label{nonneg}
Operator $L$ has no eigenvalues on $\Re \lambda \geq 0$; form $\mathcal{B}$ is negative definite.
\el

\begin{proof}
On $\Re\lambda \geq 0$, the eigenvalues of $L$ agree with those of $Mu_2:= u_2'' + f_2 u_2'$, 
through the Liouville transform $ w(x) =e^{\frac{1}{2}\int_0^{x}f_2(y)dy}u_2(x)$.
Here, we are using the fact that the essential spectra of both operators lies in 
$\{\Re \lambda <0\}\cup \{0\}$ to see that eigenfunctions on $\{\Re \lambda \geq 0\}\setminus \{0\}$
are composed of exponentially decaying modes, which, further, are in one-to-one correspondence 
in the two coordinate systems. 
This follows, in turn, from a standard theorem of Henry \cite{He} showing that the rightmost boundary of the
set of essential spectra on asymptotically constant-coefficient ordinary elliptic differential operator is
given by the rightmost boundary of the spectra of its constant-coefficient limits,
and the characterization of this boundary as the rightmost dispersion curve of the Fourier symbol of these
limits \cite{GZ,Z1}, and rightmost boundary of the set (the ``domain of consistent splitting'')
for which solutions of the eigenvalue equations either decay or grow exponentially. 
For similar arguments, see, e.g., \cite{Sa}.
Indeed, the essential spectrum of $L$ is confined to $\Re \lambda \leq -\eta<0$, since the limiting
constant-coefficient operators $L_\pm w= w''-(f_2^2(\pm \infty)/4)w$ are evidently negative definite,
$f_2$ being nonvanishing at $\pm \infty$.\footnote{This can be seen by direct computation or deduced
indirectly by the fact that the linearized traveling-wave ODE  $h'+f_2h=0$ (see discussion 
surrounding \eqref{u1eq}) admits the exponentially-decaying solution $h=H'$ at $\pm \infty$.}
At $\lambda=0$, there is a neutral, nondecaying
and nongrowing mode in the $u_2$ coordinates, but the exponentially decaying mode is still unique and in 
correspondence with that in the $w$-coordinates, hence eigenfunctions are in one-to-one correspondence
also for $\lambda=0$.

Now, introduce the ``differentiated operator'' $\mathcal{M}z= (z'+ f_2z)'$ induced by $z=u_2'$.
By divergence form of $\mathcal{M}$ we find, integrating both sides of $\mathcal{M}z=\lambda z$, that any eigenfunction
for $\Re \lambda \geq 0$ and $\lambda \neq 0$ (necessarily exponentially decaying)
has zero integral $\int_{-\infty}^{\infty}z(y)dy=0$,
hence $u_2(x):=\int_{-\infty}^x z(y)dy$ is exponentially decaying and an eigenfunction of $M$; thus, the eigenvalues
of $M$ and $\mathcal{M}$ agree on $\Re \lambda\geq 0$, $\lambda \neq 0$.

By \eqref{hident}, we have that $\bar z:=H'$ is an eigenfunction of $\mathcal{M}$ with eigenvalue $\lambda=0$.
By monotonicity $H'<0$ of the traveling wave profile, we have on the other hand that $\bar z<0$ has one sign.
Moreover, by the same computation as for $M$, the essential spectrum of $\mathcal{M}$ is 
confined to $\{\lambda:  \Re \lambda <0\} \cup \{0\}$. 
By standard Sturm-Liouville considerations, therefore- specifically, the extension to the real
line of the principal eigenvalue theorem \cite{BCJLMS,HLS}- we may conclude that $\lambda=0$ is the maximum
eigenvalue of $\mathcal{M}$. It follows that $M$ and thus $L$ have no eigenvalues on $\Re \lambda \geq 0$ other
than possibly at $\lambda=0$. Directly solving $0=Mu_2=u_2'' + f_2 u_2'$ as $u_2'=e^{-\int_0^x f_2(y)dy}u_2'(0)$,
we find that $\sgn u_2'=\sgn u'(0)$ and so $Mu_2=0$ has no nontrivial decaying solutions, and so
$\lambda=0$ is not an eigenvalue of $M$ or equivalently of $L$.
Thus, $L$ has no eigenvalues on $\Re \lambda \geq 0$, and, as remarked earlier, has essential spectrum
confined to $\Re \lambda \leq -\eta<0$.
It follows that $\mathcal{B}$ is negative definite as claimed.
\end{proof}

\br\label{posrmk}
Numerically, we find that $\frac12 f_2^2+f_2'>0$, whence $\mathcal{B}$ is negative definite by inspection.
\er

\begin{proof}[Alternate proof of Corollary \ref{smoothstab}]
	From the calculations above, \eqref{weq} is of form \eqref{geneval} with $\alpha,  \beta>0$.
Let $\lambda=ia+b$ with real $a$, $b$ and $b\ge 0$. Then, taking the imaginary part of the $L^2$ inner product
of $w$ with $Lw= \alpha \lambda w + \beta \lambda^2 w$, we have
\be
0= a\langle w, \alpha w\rangle +2ab\langle w, \beta w\rangle.
\ee
Noting that $\langle w, \alpha w\rangle +2b\langle w, \beta w\rangle>0$ for $w\not \equiv 0$,
we find therefore that $a=0$; that is, we reduce to the study of real eigenvalues $\lambda=b>0$.
Taking the real part of the $L^2$ inner product of $w$ with $Lw= \alpha b w + \beta b^2 w$, we
thus obtain 
\be
\label{smoothB}
\mathcal{B}(w,w)= b\langle w, \alpha w\rangle+ b^2\langle w, \beta w\rangle \geq 0,
\ee
with equality only if $w\equiv 0$. By negative definiteness of $\mathcal{B}$, equation \eqref{smoothB} never holds for non-zero $w$.
\end{proof}

\section{Spectral stability of discontinuous hydraulic shock profiles}
For the discontinuous case, the eigenvalue system in ``good unknowns'', after elimination of the front location, reads \cite{YZ}
\ba
\label{eigen-eq}
&(Av)_x=(E-\lambda \Id) v, \\ 
&[\lambda\overline{W}-R(\overline{W})]_\perp \cdot [Av]=0.
\ea
Since the eigenvalues $\gamma_{1,2,\pm}$ of limiting matrix $A^{-1}_\pm(E_\pm-\lambda \Id)$ satisfy
\ba 
\label{signgamma}
&\Re\gamma_{1,-}(\lambda)>0,\;\Re\gamma_{2,-}(\lambda)<0,\quad\text{for all $\Re\lambda>0,\;F<2,\;\nu>1$,}\\
&\Re\gamma_{1,+}(\lambda)>0,\;\Re\gamma_{2,+}(\lambda)>0,\quad\text{for all $\Re\lambda>0,\;\nu>\frac{1+\sqrt{1+4F}}{2F}$},
\ea
we have that $v(\lambda,x)\equiv 0$ for $x>0$, yielding $w(\lambda,x)\equiv 0$ for $x>0$. Thus, the system reduces
as described in the introduction to \eqref{syseval} on $x\in (-\infty,0)$, with boundary condition \eqref{bc} at $x=0$. 
Applying the same reduction/Liouville-type transformation as in the smooth case, we obtain the scalar second-order problem
\eqref{geneval}, with boundary condition \eqref{robin} induced by \eqref{bc}, to be computed later.

In order for $w$ to decay exponentially at $-\infty$, it is required that 
\ba 
\Re\gamma_{1,-}(\lambda)+\lim_{y\rightarrow -\infty}\frac{1}{2}\left(f_1(y)\Re \lambda+f_2(y)\right)>0.
\ea 
which follows from taking the real part of equation \eqref{ratem}.
Taking the $L^2$ inner product of $w$ with \eqref{weq} on the half line $x< 0$ yields
\be 
\label{L2equation}
\bar{w}(0)\cdot w'(0)-\langle w',w'\rangle +\Big\langle w,\Big(f_3\lambda^2+f_4\lambda-\frac{1}{4}(f_1\lambda+f_2)^2-\frac{1}{2}(f_1'\lambda+f_2')\Big)w\Big\rangle=0.
\ee 
The relation between $v(0^-)$ and $w(0)$ coordinates (for simplicity omitting ``$(0^-)$") is
\be
\label{vwrelation}
v=T_2\left[\begin{array}{c}u_1\\ u_2 \end{array}\right]=T_2\left[\begin{array}{rr}-\frac{1}{\lambda}&0\\0&1\end{array}\right]\left[\begin{array}{c}u_2'\\ u_2 \end{array}\right]=T_2\left[\begin{array}{rr}-\frac{1}{\lambda}&0\\0&1\end{array}\right]\left[\begin{array}{rr}1&-\frac{1}{2}(f_1\lambda+f_2)\\0&1\end{array}\right]\left[\begin{array}{c}w'\\ w \end{array}\right],
\ee
where $T_2$ is as in \eqref{T2}.
For a vector $\left[\begin{array}{rr}a&b\end{array}\right]^T$, define its $\perp$-vector as 
$\left[\begin{array}{c}a\\b\end{array}\right]_\perp=\left[\begin{array}{rr}b&-a\end{array}\right].$
	Substituting $[\lambda\overline{W}-R(\overline{W})]_\perp$ in \eqref{bc} and using \eqref{vwrelation} yields the equation for $w'$, $w$:
\be 
[\lambda\overline{W}-R(\overline{W})]_\perp A(0^-)T_2\left[\begin{array}{rr}-\frac{1}{\lambda}&0\\0&1\end{array}\right]\left[\begin{array}{rr}1&-\frac{1}{2}(f_1\lambda+f_2)\\0&1\end{array}\right]\left[\begin{array}{c}w'\\ w \end{array}\right]=0,
\ee 
or
\be\label{cbc}
w'(0)=(c_1+c_2\lambda)w(0),
\ee
where, after simplifying by identity
\be
\label{identity}
{H_*}^2{\left(\sqrt{H_R}+1\right)}^2+H_*H_R{\left(\sqrt{H_R}+1\right)}^2-2F^2H_R=0,
\ee
\be
\label{c1}
c_1=\frac{1}{2}f_2(H_*)-F^2\frac{H_s{\left(H_R+\sqrt{H_R}+1\right)}^2-H_R\left(2F^2+1\right)}{(\sqrt{H_R}+1)^2(H_*^3-H_s^3)},
\ee
\be 
\label{c2}
-c_2=\frac{F^2H_RH_*}{(\sqrt{H_R}+1)(H_*^3-H_s^3)}>0.
\ee

We see that \eqref{cbc} is
of form \eqref{bc}, with
$c=c_1$ and $\phi(\lambda)=c_2\lambda$.
 Substituting \eqref{cbc} in \eqref{L2equation} yields 
\be
\label{L2equations}
(c_1+c_2\lambda){w}(0)\cdot w(0)-\langle w',w'\rangle+\Big\langle w,\Big(f_3\lambda^2+f_4\lambda-\frac{1}{4}(f_1\lambda+f_2)^2-\frac{1}{2}(f_1'\lambda+f_2')\Big)w\Big\rangle=0.
\ee

\subsection{Nonexistence of complex eigenvalues}
Substituting $\lambda=ia+b$ with real $a$, $b$ in \eqref{L2equations} yields an imaginary part 
\be 
a\left(c_2{w}(0)\cdot w(0)+\Big\langle w,\Big((2f_3-\frac{1}{2}f_1^2)b+f_4-\frac{1}{2}f_1f_2-\frac{1}{2}f_1'\Big)w\Big\rangle\right)=0.
\ee
Provided $a\neq 0$ ($\lambda$ being non-real), this further simplifies to 
\be
\label{aneq0}
-c_2{w}(0)\cdot w(0)+\Big\langle w,\Big(-2f_3+\frac{1}{2}f_1^2)b-f_4+\frac{1}{2}f_1f_2+\frac{1}{2}f_1'\Big)w\Big\rangle=0,
\ee
where $-c_2$ is as in \eqref{c2}, $-f_4+\frac{1}{2}f_1f_2+\frac{1}{2}f_1'$ is as in \eqref{defico}, and 
\be
\label{sign1}
-2f_3+\frac{1}{2}f_1^2=\frac{2F^2H^5}{\left(H^3-H_s^3\right)^2}.
\ee
Substituting $\lambda=b$ with real $b$ in  \eqref{L2equations} yields
\ba  
\label{L2equations2}
&(c_1+c_2b){w}(0)\cdot w(0)-\langle w',w'\rangle \\
+&\Big\langle w,\Big((f_3-\frac{1}{4}f_1^2)b^2+(f_4-\frac{1}{2}f_1f_2-\frac{1}{2}f_1')b-\frac{1}{4}f_2^2-\frac{1}{2}f_2'\Big)w\Big\rangle=0.
\ea  
\begin{lemma}\label{halfcomplex}
On the right half space $\Re\lambda\ge 0$,
the system \eqref{eigen-eq} has no non-real eigenvalues for discontinuous hydraulic shock profiles.
\end{lemma}

\begin{proof}
	Let $\lambda=ia+b$ with real $a$, $b$ and $b\ge 0$. We first show $\lambda$ with non-vanishing $a$ is not an eigenvalue. By equation \eqref{aneq0}, it suffices to show $-c_2$, $-2f_3+\frac{1}{2}f_1^2$ and $-f_4+\frac{1}{2}f_1f_2+\frac{1}{2}f_1'$ have the same sign. Since for discontinuous hydraulic shock profiles $H_R<H_s<H_*<1$, by equations \eqref{c2}, \eqref{sign1}, we readily see $-c_2$, $-2f_3+\frac{1}{2}f_1^2$ are positive. By \eqref{defico}, it is then enough to show $f_{F,H_R}(H),H\ge H_*$ is also positive. We show in Appendix \ref{ineq}. that $H_*>H_c$ where $H_c$ defined in \eqref{Hcdef} is the positive critical point of $f_{F,H_R}$, hence 
\ba  
f_{F,H_R}(H)&>f_{F,H_R}(H_*)\\
&=2{\left(\sqrt{H_R}+1\right)}^2H_*^3-F^2H_R\left(\sqrt{H_R}+1\right)H_*+F^2{H_R}^2(1-H_*)\\
&=2{\left(\sqrt{H_R}+1\right)}^2H_*\Big(H_*^2-\frac{F^2H_R}{2{\left(\sqrt{H_R}+1\right)}}\Big)+F^2{H_R}^2(1-H_*)\\
&\ge F^2H_R^2(1-H_*)>0
\ea 
where the last inequality is because $H_*>F\sqrt{H_R}/\sqrt{2(\sqrt{H_R}+1)}$ (see Appendix \ref{ineq} for proof). 
\end{proof}

\subsection{Nonexistence of real eigenvalues}
It remains to show that $\lambda=b$ with $b>0$ real is not an eigenvalue.
Denote by $L$ the self-adjoint operator in \eqref{L} given by the $\lambda=0$ part of the lefthand side of \eqref{weq},
and 
\be\label{B2}
\mathcal{B}_2(\tilde{w},w):=c_1\bar{\tilde{w}}(0)\cdot w(0) -\langle \tilde{w}',w' \rangle - \Big\langle \tilde{w}, \Big(\frac{1}{4}f_2^2+\frac{1}{2}f_2'\Big)w\Big\rangle
\ee
the bilinear form induced on $\tilde{w},w\in H^1(\R^-)$ by $\mathcal{B}_2(\tilde{w},w):=\langle \tilde{w}, Lw\rangle_{L^2(\R^-)}$,
obtained by integration by parts under the boundary condition 
\be\label{0bc}
w'(0)=c_1w(0)
\ee
obtained by setting $\lambda=0$ in \eqref{cbc}.

\bl\label{bclem}
$c_1<\frac{1}{2}f_2(H_*)$.
\el

\begin{proof} 
See Appendix \ref{ineq}.
\end{proof}

\bc\label{halfnonneg}
The bilinear form $\mathcal{B}_2$ is negative definite.
\ec

\begin{proof}
Using again the relation \cite{He} between essential spectra of asymptotically constant-coefficient operators
and their constant-coefficient limits, we find by direct computation/analysis of the limiting Fourier symbols
that the essential spectrum of $L$ with
boundary condition \eqref{0bc} lies in $\{\Re \lambda <-\eta\leq 0\}$.
By standard calculus of variations arguments, we find therefore that either $\mathcal{B}_2$ is negative definite,
or else $\max_{|w|_{L^2(\R^-)}=1}\mathcal{B}_2(w)$ is achieved at a solution of the associated constrained
Euler-Lagrange equation, with Lagrange multiplier $\lambda$ equal to the maximum value (hence in particular real).
Moreover, the Euler-Lagrange equation is exactly the eigenvalue equation for $L$ with boundary condition \eqref{0bc}
and eigenvalue $\lambda$.
It is sufficient therefore to show that $L$ has no nonnegative eigenvalue with boundary condition \eqref{0bc},
or equivalently $Mu_2:=u_2''+ f_2 u_2'$ has no nonnegative eigenvalues with boundary condition 
\be\label{u2bc}
u_2'(0)=(c_1-f_2(0)/2) u_2(0).
	\ee

We first observe, using monotonicity of the traveling wave $H$ in a different way than in the whole-line case,
that $M$ {\it has no zero eigenvalue} for boundary condition \eqref{u2bc}, for any discontinuous hydraulic shock profile.
For, the fact that $H'$ satisfies $\bar h'+f_2 \bar h=0$ implies that $\bar u_2:= H- H_L$ is, up to a constant factor,
the unique decaying solution of $Mu_2=0$. But, by monotonicity,
$$
\frac{H'}{H-H_L} >0
$$
for all $H\in [H_*,H_L)$, in particular at $H_*=H(0)$.  Thus, $0<\bar u_2'(0)/\bar u_2(0)\neq c_1-f_2(0)/2$,
since, by Lemma \ref{bclem}, $c_1-f_2(0)/2<0$, and so there is no zero eigenfunction.
That is, as noted in the introduction, the tranlational eigenvalue at $\lambda=0$
of \eqref{syseval} has been removed by the change to ``integrated coordinates'' $(u_2, u_2')$.

More, we can modify the domain of $M$ from $(-\infty,0)$ to $(-\infty, x_0)$ for any $x_0<0$ while keeping the same 
boundary condition 
\be\label{u2modbc}
u_2'(x_0)=(c_1-f_2(0)/2)u_2(x_0),
\ee
and the same argument shows that there is no zero eigenvalue for
any choice of $x_0$.  Denote the operator acting on this modified domain by $M_{x_0}$.  
Nor are there nonzero pure imaginary eigenvalues of $M_{x_0}$, as can be seen by the Liouville transform to
a self-adjoint operator $L$ on the same domain. 
Thus, by a homotopy argument, the number of 
nonstable eigenvalues $\Re \lambda \geq 0$ of $M_{x_0}$ with boundary condition \eqref{u2modbc} is constant 
for all choices of endpoint $-\infty < x_0\leq 0$.
Shifting $x_0$ back to $0$, and taking the limit as $x_0\to -\infty$ of the associated Evans functions, we see that this number is equal to
the number of nonstable 
eigenvalues of the limiting constant-coefficient operator
$$
M_\infty u_2:= u_2'' + f_2(-\infty) u_2'
$$
with boundary condition \eqref{u2bc} imposed at $x=0$.  But this may be seen by direct calculation to be zero,
since decaying solutions for $\Re \lambda \geq 0$ are of form $e^{\mu x}$ with $\mu$ real and positive,
and so  $0< \mu =u_2'(0)/u_2(0)\neq c_1-f_2(0)/2<0$, contradicting existence of a decaying eigenfunction.

Thus, there are no nonnegative eigenvalues of $M$ with boundary condition \eqref{u2bc},  hence no 
nonnegative eigenvalue of $L$ with boundary condition \eqref{0bc},
and it follows that $\mathcal{B}_2(w)<0$ as claimed.
\end{proof}

\br\label{diffrmk}
The argument in the half-line case, though still based on monotonicity,
is essentially different from the standard principal eigenvalue argument used in the whole line case,
using a homotopy not available there.  So far as we know, this approach is new.
\er

\br\label{posrmk2}
Numerically, $\frac14 f_2^2+\frac12f_2'>\delta >0$, $c_1<0$, whence $\mathcal{B}_2$ is negative definite by inspection.
\er

\begin{corollary}\label{discontstab}
All discontinuous hydraulic shock profiles are weakly 
spectrally stable in the sense that system \eqref{syseval} has no eigenvalue $\lambda$ with $\Re \lambda \geq 0$
and $\lambda \neq 0$.
\end{corollary}

\begin{proof}
Lemma \ref{halfcomplex}, it is sufficient to consider real eigenvalues $\lambda = b>0$.
Taking the real part of the $L^2$ inner product of $w$ with $Lw= \alpha \lambda w + \beta \lambda^2 w$ on $\R^-$, we
obtain 
\be\label{temp1}
\phi(b)|w(0)|^2+ \mathcal{B}_2(w)= b\langle w, \alpha w\rangle+ b^2\langle w, \beta w\rangle \geq 0,
\ee
with equality only if $w\equiv 0$. By negative definiteness of $\mathcal{B}_2$, and negativity of
$\phi(b)$ ($c_2<0$), 
equation \eqref{temp1} never holds for $b>0$, proving the result.
\end{proof}

\appendix
\section{Proof of inequalities}\label{ineq}
Let $H_R=1/\nu^2$ with $\nu>1$. We establish the following inequalities.

\medskip

{\bf 1. $H_*>H_c$}:
\ba 
&&H_*&>H_c\\
&\Leftrightarrow& \frac{-\nu-1+\sqrt{8F^2\nu^4+\nu^2+2\nu+1}}{2\nu^2\left(\nu+1\right)}&>\frac{F\sqrt{\nu^2+\nu+1}}{\sqrt{6}\nu(\nu+1)}\\
&\Leftrightarrow&\sqrt{8F^2\nu^4+\nu^2+2\nu+1}&>\frac{2F\nu\sqrt{\nu^2+\nu+1}}{\sqrt{6}} +\nu+1\\
&\Leftrightarrow& 8F^2\nu^4&>\frac{2}{3}F^2(\nu^4+\nu^3+\nu^2)+\frac{4F\nu(\nu+1)\sqrt{\nu^2+\nu+1}}{\sqrt{6}}\\
&\Leftrightarrow&\frac{11}{3}F\nu^3-\frac{1}{3}F\nu^2-\frac{1}{3}F\nu&>\frac{2(\nu+1)}{\sqrt{6}}\sqrt{\nu^2+\nu+1}\\
&\Leftrightarrow& F^2\left(\frac{121}{9}\nu^6-\frac{22}{9}\nu^5-\frac{7}{3}\nu^4+\frac{2}{9}\nu^3+\frac{1}{9}\nu^2\right)&>\frac{2}{3}(\nu^2+2\nu+1)(\nu^2+\nu+1).
\ea 
Because on the domain of discontinuous hydraulic shock profiles, $F>\frac{1}{\nu^2}+\frac{1}{\nu}$. Thus, we have 
\ba 
&F^2\left(\frac{121}{9}\nu^6-\frac{22}{9}\nu^5-\frac{7}{3}\nu^4+\frac{2}{9}\nu^3+\frac{1}{9}\nu^2\right)-\frac{2}{3}(\nu^2+2\nu+1)(\nu^2+\nu+1)\\
>&\left(\frac{1}{\nu^2}+\frac{1}{\nu}\right)^2\left(\frac{121}{9}\nu^6-\frac{22}{9}\nu^5-\frac{7}{3}\nu^4+\frac{2}{9}\nu^3+\frac{1}{9}\nu^2\right)-\frac{2}{3}(\nu^2+2\nu+1)(\nu^2+\nu+1)\\
>&\left(\frac{1}{\nu^2}+\frac{1}{\nu}\right)^2\left(\frac{78}{9}\nu^6+\frac{2}{9}\nu^3+\frac{1}{9}\nu^2\right)-\frac{2}{3}(\nu^2+2\nu+1)(\nu^2+\nu+1)\\
=&\frac{{\left(\nu+1\right)}^2\left(72\nu^4-6\nu^3-6\nu^2+2\nu+1\right)}{9\nu^2}>0.
\ea 

\medskip

{\bf 2. $H_*>F\sqrt{H_R}/\sqrt{2(\sqrt{H_R}+1)}$}: 

\ba 
&&\frac{-\nu-1+\sqrt{8F^2\nu^4+\nu^2+2\nu+1}}{2\nu^2\left(\nu+1\right)}&>\frac{F}{\sqrt{2\nu(\nu+1)}}\\
&\Leftrightarrow& \sqrt{8F^2\nu^4+\nu^2+2\nu+1}&>F\sqrt{2\nu^3(\nu+1)}+\nu+1\\
&\Leftrightarrow& 3F\nu^3-F\nu^2&>(\nu+1)\sqrt{2\nu(\nu+1)}\\
&\Leftrightarrow& 9F^2\nu^6-6F^2\nu^5+F^2\nu^4&>2(\nu+1)^3\nu.
\ea 
Again, using $F>\frac{1}{\nu^2}+\frac{1}{\nu}$, we have 
\ba 
&9F^2\nu^6-6F^2\nu^5+F^2\nu^4-2(\nu+1)^3\nu>F^2(3\nu^6+\nu^4)-2(\nu+1)^3\nu\\
>&\left(\frac{1}{\nu^2}+\frac{1}{\nu}\right)^2(3\nu^6+\nu^4)-2(\nu+1)^3\nu=\left(\nu^2-1\right)^2>0.
\ea

\medskip

{\bf 3. $c_1<\frac{1}{2}f_2(H_*)$}:

It suffices to show 
\be
H_*{\left(H_R+\sqrt{H_R}+1\right)}^2-H_R\left(2F^2+1\right)>0
\ee
Let $\tilde{\nu}=\frac{1}{\nu}$, replace $H_*$, $H_R$ by $-\frac{\tilde{\nu}\,\left(\tilde{\nu}+\tilde{\nu}^2-\sqrt{8\,F^2+\tilde{\nu}^4+2\,\tilde{\nu}^3+\tilde{\nu}^2}\right)}{2\,\left(\tilde{\nu}+1\right)}$ and $\tilde{\nu}^2$, it then suffices to show for $\tilde{\nu}+\tilde{\nu}^2<F<2$, $0<\tilde{\nu}<1$:
\ba 
&&&(\tilde{\nu}^4+2\,\tilde{\nu}^3+3\,\tilde{\nu}^2+2\,\tilde{\nu}+1)\sqrt{8\,F^2+\tilde{\nu}^4+2\,\tilde{\nu}^3+\tilde{\nu}^2}\\
&&>&4\,F^2\,\tilde{\nu}^2+4\,F^2\,\tilde{\nu}+\tilde{\nu}^6+3\,\tilde{\nu}^5+5\,\tilde{\nu}^4+5\,\tilde{\nu}^3+5\,\tilde{\nu}^2+3\,\tilde{\nu}\\
&\Leftrightarrow&&\#(F,\tilde{\nu}):=-16\,F^4\,\tilde{\nu}^2\,{\left(\tilde{\nu}+1\right)}^2+F^2\,\left(16\,\tilde{\nu}^6+48\,\tilde{\nu}^5+72\,\tilde{\nu}^4+64\,\tilde{\nu}^3+56\,\tilde{\nu}^2+32\,\tilde{\nu}+8\right)\\
&&&-4\,\tilde{\nu}^2\,{\left(\tilde{\nu}+1\right)}^2\,\left(\tilde{\nu}^4+2\,\tilde{\nu}^3+3\,\tilde{\nu}^2+2\,\tilde{\nu}+2\right)>0.
\ea 
As a quadratic function of variable $F^2$, the axis of symmetry of $\#$ is always on the right half plane, we then examine values of $\#$ at end points $F=\tilde{\nu}+\tilde{\nu}^2$ and $F=2$.
\ba
\#(\tilde{\nu}+\tilde{\nu}^2,\tilde{\nu})&=4\,\tilde{\nu}^3\,\left(1-\tilde{\nu}\right)\,{\left(\tilde{\nu}+1\right)}^3\,\left(4\,\tilde{\nu}^5+16\,\tilde{\nu}^4+24\,\tilde{\nu}^3+20\,\tilde{\nu}^2+11\,\tilde{\nu}+6\right)>0,\\
\#(2,\tilde{\nu})&=4\,{\left(\tilde{\nu}^2+\tilde{\nu}-2\right)}^2\,\left(-\tilde{\nu}^4-2\,\tilde{\nu}^3+9\,\tilde{\nu}^2+10\,\tilde{\nu}+2\right)\\
&>4\,{\left(\tilde{\nu}^2+\tilde{\nu}-2\right)}^2\,\left(6\,\tilde{\nu}^2+10\,\tilde{\nu}+2\right)>0.
\ea
Therefore, 
\be
\#(F,v)>\min\left(\#(\tilde{\nu}+\tilde{\nu}^2,\tilde{\nu}),\#(2,\tilde{\nu})\right)>0.
\ee

\end{document}